\documentclass[twoside,11pt]{article}

\usepackage{mathrsfs}

\usepackage{comment}
\usepackage{amsthm}
\usepackage{amsmath,amsfonts,amssymb}
\usepackage{graphics,color}
\usepackage{enumerate}
\usepackage{float}

\usepackage{authblk}

\newtheoremstyle{mytheorem}{4pt}{4pt}{\itshape}{}{\scshape}{.}{.5em}{}
\theoremstyle{mytheorem}
\newtheorem{theorem}{Theorem}

\newtheorem{lemma}{Lemma}

\newtheorem{corollary}{Corollary}
\newtheorem{definition}{Definition}
\newtheorem{remark}{Remark}

\newcommand*{\R}{\ensuremath{\mathbb{R}}}
\renewcommand*{\P}{\ensuremath{\mathcal{P}}}

\newcommand*{\T}{\ensuremath{\mathbb{T}}}
\newcommand*{\N}{\ensuremath{\mathbb{N}}}
\newcommand*{\Q}{\ensuremath{\mathbb{Q}}}
\newcommand{\Z}{\ensuremath{\mathbb{Z}}}

\newcommand*{\supp}{\ensuremath{\mathrm{supp\,}}}
\newcommand*{\rank}{\ensuremath{\mathrm{rank\,}}}

\newcommand*{\cof}{\ensuremath{\mathrm{cof\,}}}

\renewcommand*{\d}{\ensuremath{\text{\,d}}}
\newcommand{\n}{\ensuremath{\mathbf{n}}}
\renewcommand{\a}{\ensuremath{\mathbf{a}}}

\begin{document}

\title{Laminates supported on cubes}

\author[1]{Gabriella Sebesty\'en}
\author[2]{ L\'aszl\'o Sz\'ekelyhidi Jr.}

\affil[1]{E\"otv\"os Lor\'ant University, Budapest}

\affil[2]{Universit\"at Leipzig}


\maketitle

\begin{abstract}
In this paper we study the relationship between rank-one convexity and quasiconvexity in the space of $2\times 2$ matrices.
We show that a certain procedure for constructing homogeneous gradient Young measures from periodic deformations, that arises from V.~\v Sver\'ak's celebrated counterexample in higher dimensions, always yields laminates in the $2\times 2$ case. 
\end{abstract}


\section{Introduction}

A continuous function $f:\R^{m\times n}\to\R$ is said to be {\it quasiconvex} if 
\begin{equation}\label{e:qc}
\int_{\T^n}f\bigl( A+Du(x)\bigr)\,dx\geq f(A)
\end{equation}
for any matrix $A\in \R^{m\times n}$ and any Lipschitz function $u:\R^n\to\R^m$ which is periodic with respect to the lattice $\Z^n$. Here $\T^n$ is the unit cube of $\R^n$. Equivalently, quasiconvexity may be defined by using smooth periodic or smooth compactly supported test functions $u\in C_c^{\infty}(\Omega;\R^m)$ for any bounded domain with Lipschitz boundary (see \cite{Morrey:1952we,Sverak:1992jp,Muller:1996wb}). It is well known \cite{Morrey:1952we} that quasiconvexity 
of $f$ is equivalent to the weak star lower-semicontinuity of the functional $u\mapsto \int_{\Omega}f\bigl( Du(x)\bigr)\,dx$ in $W^{1,\infty}(\Omega;\R^m)$.

Because of its fundamental importance in the calculus of variations, it is of interest to find necessary and sufficient conditions for quasiconvexity. The most well-known necessary condition is the rank-one convexity of $f$, namely that for any $A,B\in\R^{m\times n}$ with $\rank(B)=1$ 
$$
t\mapsto f(A+tB)\quad\textrm{ is convex.}
$$
This arises by using test functions of the form
\begin{equation}\label{e:test1}
u(x)=\a h(x\cdot \n),
\end{equation}
where $\a\in \R^m$, $\n\in \Z^n$ and $h:\R\to\R$ is the 1-periodic extension of the saw-tooth function
$$
h(t)=\begin{cases} t& \textrm{ for }0\leq t\leq 1/2,\\ 1-t& \textrm{ for } 1/2\leq t\leq 1.\end{cases}
$$
Indeed, by direct calculation 
$$
Du(x)=h'(x\cdot \n) \a\otimes \n\quad\textrm{ for a.e. $x$}
$$
and 
$$
\int_{\T^n}f\bigl( A+Du(x)\bigr)\,dx=\frac{1}{2}f(A+\a\otimes \n)+\frac{1}{2}f(A-\a\otimes \n).
$$
Thus, \eqref{e:qc} implies that
\begin{equation}\label{e:rc}
t\mapsto f(A+t\a\otimes \n)\textrm{ is convex}
\end{equation}
for any $\a\in R^m$ and $\n\in \Z^n$. Since for any $\n\in\Q^n$ there exists $\lambda\neq 0$ such that $\lambda\n\in\Z^n$, by using that $\a\otimes\n=(\tfrac{1}{\lambda}\a)\otimes (\lambda\n)$, one can easily extend \eqref{e:rc} to all $\n\in \Q^n$, and then, using the continuity of $f$, to all $\n\in \R^n$. Thus $f$ is rank-one convex. 

The question whether the converse implication holds, i.e. whether rank-one convexity is also sufficient for quasiconvexity, has attracted a lot of attention since Morrey's seminal paper \cite{Morrey:1952we}, not only because of the relevance to the calculus of variations, but also because of surprising and deep connections to other areas \cite{AlBaernstein:1999wx}. In the case where $m\geq 3$, V.~\v Sver\'ak constructed in \cite{Sverak:1992cl} an ingenious example showing that rank-one convexity is not sufficient for quasiconvexity. The case $m=n=2$, however, remains wide open. Indeed, there is evidence that for this case rank-one convexity might be sufficient after all \cite{Pedregal:1996di,Pedregal:1998vt,AlBaernstein:1999wx,Faraco:2008vo,Astala:2012ex}. 

Returning to necessary conditions for quasiconvexity, consider now test functions of the form
\begin{equation}\label{e:test2}
u(x)=\sum_{i=1}^N\a_i h(x\cdot \n_i+c_i),
\end{equation}
where $\a_i\in \R^m$, $\n_i\in \Z^2$ and $c_i\in\R$. As above, for a.e. $x\in\T^n$ 
$$
Du(x)=\sum_{i=1}^Nh'(x\cdot \n_i+c_i)\,\a_i\otimes \n_i=\sum_{i=1}^N\epsilon_i(x)\, \a_i\otimes \n_i,
$$
where $\epsilon_i(x)\in\{-1,+1\}$ for each $i$. The set of possible values of $Du(x)$ are precisely the $2^N$ vertices of a {\it rank-one hypercube}, i.e.~an $N$-dimensional cube immersed in $\R^{m\times n}$, whose sides are given by the rank-one matrices  
$C_i:=\a_i\otimes \n_i$, $i=1,\dots,N$. Let us denote the vertices by 
$$
X_{\epsilon}=\sum_{i=1}^N\epsilon_iC_i,\quad \epsilon\in\{-1,+1\}^N.
$$
The integral in \eqref{e:qc} defines a probability measure $\nu$ supported on the vertices of this hypercube - in fact $\nu$ is a homogeneous gradient Young measure, see \cite{Kinderlehrer:1991vb,Sverak:1992jp,Muller:1996wb}, with barycenter $0$. More precisely,
\begin{equation}\label{e:nu}
\int_{\T^n}f\bigl(Du(x)\bigr)\,dx=\sum_{\epsilon\in\{-1,+1\}^N}\nu_\epsilon f(X_{\epsilon}),
\end{equation}
so that a consequence of the quasiconvexity of $f$ would be the inequality
\begin{equation}\label{e:qc-test}
\sum_{\epsilon\in\{-1,+1\}^N}\nu_\epsilon f(X_{\epsilon})\geq f(0).
\end{equation}
Our aim in this paper is to analyse in more detail this inequality in the case $n=m=2$.
As a first observation, note that the weights $\nu_{\epsilon}$ are determined by the volume fractions in $\T^n$ where the functions  
$h'(x\cdot \n_i+c_i)$ each take the value $\pm 1$ respectively. In particular it does not depend on the choices of the vectors $\a_i$. 

It turns out that with $N=2$ nothing more is gained from \eqref{e:test2} with respect to \eqref{e:test1} - see Lemma \ref{l:N=2} below. Indeed, with $N=2$ the inequality \eqref{e:qc-test} becomes 
$$
\frac{1}{4}\Bigl(f(X_{++})+f(X_{+-})+f(X_{-+})+f(X_{--})\Bigr)\geq f(0),
$$
which is clearly satisfied by all rank-one convex functions.
The situation, however, becomes much more interesting if $N\geq 3$. Indeed, 
the example of \v Sver\'ak can be understood, following R.~James' modification (see Section 4.7 in \cite{Muller:1996wb}), precisely in this way. To this end let 
$N=3$, and set $\n_1=(1,0)$, $\n_2=(0,1)$, $\n_3=(1,1)$ and the phases are $c_1=c_2=0$ and $c_3=1/4$. A direct calculation (see e.g. \cite{Muller:1996wb}) easily shows that the measure $\nu$ in \eqref{e:nu} is given in this case by
\begin{equation}\label{e:nu1-3}
\begin{split}
\nu_{+++}=\nu_{+--}=\nu_{-+-}=\nu_{--+}&=1/16,\\
\nu_{---}=\nu_{++-}=\nu_{+-+}=\nu_{-++}&=3/16.
\end{split}
\end{equation}
Whether any rank-one convex function satisfies the corresponding inequality \eqref{e:qc-test} now requires specific knowledge of the vectors $\a_i$. Indeed, for the $3\times 2$ case, where $\a_1=(1,0,0)$, $\a_2=(0,1,0)$ and $\a_3=(0,0,1)$, there exist rank-one convex functions which do not satisfy \eqref{e:qc-test} (\cite{Sverak:1992cl}). 

On the other hand, the same example does not work in the $2\times 2$ case: if $\a_1=(1,0)$, $\a_2=(0,1)$ and $\a_3=(1,1)$, it was shown by P.~Pedregal in \cite{Pedregal:1996di} (see also \cite{Pedregal:1997vs,Pedregal:1998vt}) that every rank-one convex function satisfies \eqref{e:qc-test} - in other words the measure $\nu$ in \eqref{e:nu1-3} is a laminate (see \cite{Pedregal:1993wu} and below for definitions). It was recently suggested in \cite{Pedregal:2014up} that for better choices of $\a_i\in\R^2$ the measure $\nu$ might not be a laminate. Our main result in this paper is to show that this is not the case:

\begin{theorem}\label{t:main}
Given any $C_1,C_2,C_3$ rank-1 matrices in $\R^{2\times 2}$, for any rank-one convex function $f:\R^{2\times 2}\to\R$ we have
\begin{equation}\label{e:main}
\begin{split}
f(0)\leq&\frac{1}{16}\bigl(f(X_{+++})+f(X_{+--})+f(X_{-+-})+f(X_{--+})\bigr)+\\
&\frac{3}{16}\bigl(f(X_{---})+f(X_{++-})+f(X_{+-+})+f(X_{-++})\bigr),
\end{split}
\end{equation}
where $X_{\pm\pm\pm}$ denotes the matrix $\pm C_1\pm C_2\pm C_3$.
\end{theorem}

A different generalization of \cite{Pedregal:1998vt} was analysed in \cite{Bandeira:2011tx} - here the rank-one cube is the same, but the barycenter of the measure is varied.


\section{Interaction of frequencies}

In this section we take a closer look at those probability measures $\nu$ that arise from the construction in \eqref{e:nu} with $u$ given by \eqref{e:test2} and $n=m=2$. Note that for $2\times 2$ matrices the determinant is a quadratic function with the identity
\begin{equation}\label{e:det}
\det(X+Y)=\det X+\langle \cof X,Y\rangle +\det Y,
\end{equation}
where $\cof X$ denotes the cofactor matrix and $\langle X,Y\rangle=\sum_{1\leq i,j\leq 2}X_{ij}Y_{ij}$ is the standard Hilbert-Schmidt scalar product on matrices.  

\medskip

First of all, since $u:\T^2\to\R^2$ in \eqref{e:test2} is a periodic Lipschitz function, integration by parts shows that
$$
\int_{\T^2}Du(x)\,dx=0,\,\int_{\T^2}\det Du(x)\,dx=0.
$$
Consequently
\begin{equation}\label{e:pc-constraints}
\sum_{\epsilon\in\{-1,+1\}^N}\nu_\epsilon\,X_{\epsilon}=0,\\ \sum_{\epsilon\in\{-1,+1\}^N}\nu_\epsilon\,\det(X_{\epsilon})=0.
\end{equation}
From these equations we can deduce the following simple lemma:
\begin{lemma}\label{l:N=2}
If $N=2$ and $\n_1,\n_2$ are linearly independent vectors, then the measure $\nu$ satisfies
$$
\nu_{++}=\nu_{+-}=\nu_{-+}=\nu_{--}=1/4.
$$
\end{lemma}

\begin{proof}
Since $\nu$ is independent of $\a_i$, let us set without loss of generality $\a_i=\n_i$, $i=1,2$, and $C_i=\n_i\otimes \n_i$. 
Hence $C_1,C_2$ are linearly independent (in $\R^{2\times 2}$) and, working in coordinates $(x,y)\sim xC_1+yC_2$, 
the first equation in \eqref{e:pc-constraints} leads to 
\begin{eqnarray*}
\nu_{++}+\nu_{+-}-\nu_{-+}-\nu_{--}&=&0,\\
\nu_{++}-\nu_{+-}+\nu_{-+}-\nu_{--}&=&0.
\end{eqnarray*}
Next, a quick calculation based on \eqref{e:det} shows that 
$$
\det(C_1\pm C_2)=\pm\langle \cof C_1,C_2\rangle=\pm(\n_1\cdot \n_2^\perp)^2\neq 0.
$$
Therefore the second equation in \eqref{e:pc-constraints} leads to 
\begin{equation*}
\nu_{++}-\nu_{+-}-\nu_{-+}+\nu_{--}=0.
\end{equation*}
Finally, since $\nu$ is a probability measure, we also have
\begin{equation*}
\nu_{++}+\nu_{+-}+\nu_{-+}+\nu_{--}=1.
\end{equation*}
It is easy to check that the four equations we obtained for $\nu_{\pm\pm}$ has the unique solution
$$
\nu_{++}=\nu_{+-}=\nu_{-+}=\nu_{--}=1/4.
$$
\end{proof}

Next, let us consider the situation where $N=3$ in \eqref{e:test2}. 
Using the periodicity of $f$, we may assume without loss of generality that $c_1=c_2=0$, $c_3=c$. Moreover, we will assume that no two of the vectors $\n_1,\n_2,\n_3$ are collinear. The measure $\nu$ defined in \eqref{e:nu} is now supported on the $8$ vertices of a rank-one cube. Using Lemma \ref{l:N=2} we see that
the sum of the $2$ weights on any edge of the cube is equal to $1/4$, see Figure \ref{f:symmetric}.

\begin{figure}[H]
\begin{center}
\input{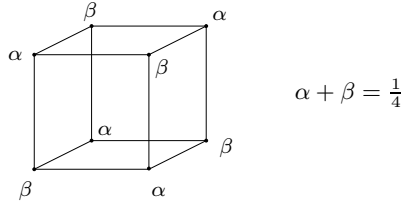}
\end{center}
\caption{Symmetric measures}\label{f:symmetric}
\end{figure}

This motivates the following
\begin{definition}[Symmetric measures]\label{d:symmetric}
A probability measure $\nu$ supported on the vertices of a cube is said to be {\it symmetric} if the weights satisfy 
\begin{eqnarray*}
\nu_{+++}=\nu_{+--}=\nu_{-+-}=\nu_{--+}&=&\alpha,\\
\nu_{---}=\nu_{++-}=\nu_{+-+}=\nu_{-++}&=&\beta,
\end{eqnarray*}
with $\alpha+\beta=1/4$.
\end{definition}

Returning to the formula \eqref{e:nu} we see that the weights $\alpha,\beta$ can be obtained from calculating
$$
\int_{\T^2}\chi(x)\,dx=4(\alpha-\beta),
$$
where $\chi(x)=\prod_{i=1}^3h'(\n_i\cdot x+c_i)$. Using the periodicity and an affine change of variables, we may then assume
that $\n_1=(1,0)$, $\n_2=(0,1)$ and $\n_3\in\Z^2$. As a consequence of Lemma \eqref{l:N=2}, it then suffices to calculate
$$ 
\int_{Q}h'(\n_3\cdot x+c)\,dx=\alpha-\beta,
$$
where $Q=(0,1/2)\times (0,1/2)$. In the following lemma we perform this calculation, however for simplicity we rescale $Q$ to be the unit square.

\begin{lemma}
Let $\n=(k,l)$ with $k,l\in\N$, $Q=(0,1)^2$, and let $f$ be the 2-periodic function
$$
f(t)=\begin{cases} +1&\text{ if }t\in [0,1)\\
                   -1&\text{ if }t\in [1,2)
     \end{cases}.
$$
If either $k$ or $l$ is even, then $\int_Qf(x\cdot \n+c)\d x=0$. Otherwise
$$
\max_c\int_Q f(x\cdot \n+c)\d x=\frac{1}{2kl}.
$$
\end{lemma}

\begin{figure}[H]
\begin{center}
\input{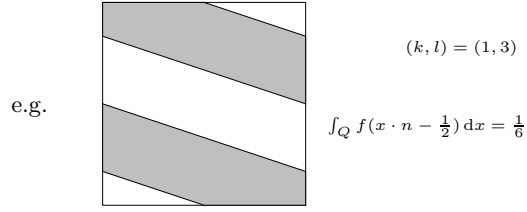}
\end{center}
\caption{Calculating volume fractions}\label{f:volume}
\end{figure}

\begin{proof}
Note that $f(t)=f(t+2)$ for all $t$, and $\int_If(t)\d t=0$ for any interval $I$ of length 2. For any $c$ we have
$$
\int_0^1f(kx_1+c)\d x_1=\frac{1}{k}\int_{c}^{k+c}f(t)\d t,
$$
and so if $k$ is even, the above integral is zero. Then Fubini gives
$\int_Qf(x\cdot \n)\d x=0$
whenever $k$ or $l$ is even. Moreover, if $k$ is odd, then
$$
\frac{1}{k}\int_{c}^{k+c}f(t)\d t=
\frac{1}{k}\bigl(\int_{c}^{c+1}+\int_{c+1}^{c+k}\bigr)f(t)\d t=
\frac{1}{k}\int_{c}^{c+1}f(t)\d t=:g(c),
$$
where $g(c)=g(c+2)$ and $g(c)+g(c+1)=0$. Furthermore, it is easy to verify that 
$$
g(c)=\begin{cases}\frac{1}{k}(1-2c)\text{ if }c\in[0,1)\\
                  \frac{1}{k}(2c-3)\text{ if }c\in[1,2)
     \end{cases}.
$$
\begin{figure}[H]
\begin{center}
\input{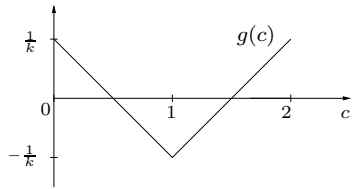}
\caption{$\int_0^{1}f(kx_1+c)\d x_1$ is a 2-periodic function of $c$.}
\end{center}
\end{figure}
\noindent
Note that $\int_Ig(s)\d s=0$ for any interval of length 2. Now using Fubini 
$$
I(c)=\int_Qf(x\cdot \n+c)\d x=\int_0^1g(lx_2+c)\d x_2=
\frac{1}{l}\int_c^{c+1}g(s)\d s
$$
by the same argument as before. So $I$ is 2-periodic and $I(c)+I(c+1)=0$.
Now let $c\in (0,1)$. Then
$$
I(c)=\frac{1}{kl}\biggl(\int_c^1(1-2s)\d s+
\int_1^{1+c}(2s-3)\d s\biggr)=\frac{2c(c-1)}{kl}.
$$
\begin{figure}[H]
\begin{center}
\input{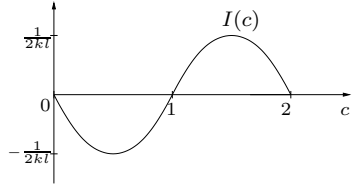}
\caption{$I(c)=\int_Qf(x\cdot \n+c)\d x$ is a 2-periodic function of $c$.}
\end{center}
\end{figure}

\end{proof}

Summarizing the results of this section we then obtain:
\begin{corollary}
If $u:\T^2\to\R^2$ is given by
$$
u(x)=\sum_{i=1}^3h'(\n_i\cdot x+c_i)
$$
for some $\n_i,c_i$ and $\nu$ is the probability measure obtained from the formula \eqref{e:nu}, then
$\nu$ is a symmetric probability measure supported on 
$\{\sum_{i=1}^3\pm \a_i\otimes \n_i\}$ with
\begin{eqnarray*}
\nu_{+++}=\nu_{+--}=\nu_{-+-}=\nu_{--+}&=&\alpha,\\
\nu_{---}=\nu_{++-}=\nu_{+-+}=\nu_{-++}&=&\beta,
\end{eqnarray*}
such that $\alpha+\beta=1/4$ and $-1/8\leq \alpha-\beta\leq 1/8$.

where $\gamma\in [-\frac{1}{16kl},\frac{1}{16kl}]$ if $k$ and $l$ are both odd, or $\gamma=0$ if $k$ or $l$ is even.
\end{corollary}

In particular if $\n_1=(1,0)$, $\n_2=(0,1)$ and $\n_3=(k,l)$ ($k,l\in\Z$), then 
$\alpha-\beta\in [-\frac{1}{8kl},\frac{1}{8kl}]$ if $k$ and $l$ are both odd, and $\alpha=\beta$ if $k$ or $l$ is even.
The extremal cases where $|\alpha-\beta|=1/8$ arise from $\n_1=(1,0)$, $\n_2=(0,1)$ and $\n_3=(1,1)$.

There is one parameter 
giving symmetric measures, either $\alpha-\beta$ as above, or, equivalently, 
$$
\frac{\nu(X_0)}{\nu(X_1)}=\frac{\alpha}{\beta}\in [1/3,3].
$$
In the rest of this paper we show that for any ratio $\alpha/\beta\in [1/3,3]$ the corresponding symmetric measures laminates. As a consequence 
no counterexample to the equivalence of rank-one convexity and quasiconvexity can arise from such configurations.


\section{Laminates and semiconvex hulls}

Let $\P(\R^{2\times 2})$ denote the set of all compactly supported probability measures on $\R^{2 \times 2}$.
For $\nu \in \P$ we denote by $\overline{\nu} = \int_{\R^{2\times 2}}{X d\nu(X)}$ the center of mass or \textit{barycenter} of $\nu.$

\begin{definition}\label{d:laminate}
A measure $\nu \in \P$ is called a \textit{laminate}, denoted $\nu\in\mathcal{L}$, if 
\begin{equation}\label{e:deflam}
f(\overline{\nu}) \leq \int_{\R^{2\times 2}}{f d\nu}
\end{equation} 
for all rank-one convex functions $f$. The set of laminates with barycenter $0$ and supported in a compact set $K\subset\R^{2\times 2}$ is denoted by $\mathcal{L}_0(K)$. 
\end{definition}
With this definition, the question of whether an inequality of the type \eqref{e:qc-test} holds for all rank-one convex functions amounts to the question of whether $\nu$ is a laminate. We note in passing that $\mathcal{L}_0(K)$ is a convex set.

\begin{definition}\label{d:prelam}
We call $\mathcal{PL}(\R^{2\times 2})$ the set of \emph{prelaminates}. This is the smallest
class of probability measures on $\R^{2\times 2}$ which
\begin{itemize}
\item contains all measures of the form $\lambda \delta_A+(1-\lambda)\delta_B$ with $\lambda\in [0,1]$ and $\textrm{rank}(A-B)=1$;
\item is closed under splitting in the following sense: if $\lambda\delta_A+(1-\lambda)\tilde\nu$ belongs to $\mathcal{PL}(\R^{2\times 2})$ for some $\tilde\nu\in\P(\R^{2\times 2})$ and $\mu$ also belongs to $\mathcal{PL}(\R^{2\times 2})$ with $\overline{\mu}=A$, then also $\lambda\mu+(1-\lambda)\tilde\nu$ belongs to $\mathcal{PL}(\R^{2\times 2})$.
\end{itemize}
\end{definition} 
The order of a prelaminate denotes the number of splittings required to obtain the measure from a Dirac measure. It is clear from the definition that $\mathcal{PL}(\R^{2\times 2})$ consists of atomic measures. Also, from a repeated application of Jensen's inequality it follows that $\mathcal{PL}\subset\mathcal{L}$. Furthermore, $\nu\in\mathcal{L}$ if and only if there exists a sequence $\nu_k\in \mathcal{PL}$ with uniformly bounded support, such that $\nu_k\overset{*}{\rightharpoonup}\nu$ (see \cite{Pedregal:1993wu}).

We remind the reader that prelaminates, as defined above, are precisely those probability measures $\nu=\sum_{i=1}^N\lambda_i\delta_{X_i}$ for which the sequence $\{(\lambda_i,X_i)\}_{1\leq i\leq N}$ satisfies the 
$(H_N)$-condition (see \cite{Pedregal:1993wu}). It is also worth noting that in \cite{Pedregal:1996di,Pedregal:1998vt}
the inequality \eqref{e:main} is verified for rank-one convex functions in the case
$$
\n_1=\a_1=(1,0),\,\n_2=\a_2=(0,1),\,\n_3=\a_3=(1,1)
$$
by showing that the associated measure $\nu$ is a convex combination of 3 prelaminates of order 6. 

\bigskip

Next, we introduce the various semiconvex hulls. Let $K\subset\R^{2\times 2}$ be a compact set. The lamination-convex hull is defined as follows. We first set $K^{lc,0}:=K$, and for any $i\geq 0$
$$
K^{lc,i+1}:=\left\{\lambda X+(1-\lambda)Y:\,X,Y\in K^{lc,i},\,\rank(X-Y)=1,\,\lambda\in [0,1]\right\}.
$$
Then, the lamination-convex hull is defined as $K^{lc}=\bigcup_{i\geq 0}K^{lc,i}$. The rank-one convex hull is defined in terms of separation with rank-one convex functions:
$$
K^{rc}:=\left\{X:\,f(X)\leq \sup_{K}f\quad\textrm{ for all rank-one convex $f$}\right\}.
$$
It is not difficult to see (e.g. \cite{Muller:1996wb,Matousek:1998wh,Kirchheim:2003wp}) that $K^{rc}\supset K^{lc}$, but equality does not necessarily hold. Moreover, another characterization of $K^{rc}$ follows from duality:
$$
K^{rc}=\left\{\bar{\nu}:\,\nu\textrm{ is a laminate with }\supp\nu\subset K\right\}.
$$
Finally, the polyconvex hull of $K$ is defined as
$$
K^{pc}=\left\{\bar{\nu}:\,\nu\in\P(\R^{2\times 2})\textrm{ with }\supp\nu\subset K\textrm{ and }\int\det(X)d\nu(X)=\det(\bar{\nu})\right\}
$$
Since the functions $X\mapsto \pm\det X$ are rank-one convex, we see that $K^{rc}\subset K^{pc}$.

\medskip

For calculating lamination hulls, the following will be useful:
\begin{lemma}\label{L:quadraticsurface}
Suppose $K=\{X_1,X_2,X_3,X_4\}\subset\R^{2\times 2}$ is a {\it rank-one square}, in other words suppose
$$
\det(X_1-X_2)=0,\,\det(X_2-X_3)=0,\\
\det(X_3-X_4)=0,\,\det(X_4-X_1)=0.
$$
Then $K^{lc}=K^{pc}$. For the hull there are three cases depending on the 
determinant of the `diagonals':
\begin{enumerate}
\item If $\det(X_1-X_3)=0$ and $\det(X_2-X_4)\neq 0$, then $X_1$, $X_2$, $X_3$ and $X_1$, $X_3$, $X_4$ both lie on a rank-one plane, and
$$
K^{lc}=\{X_1,X_2,X_3\}^{co}\cup\{X_1,X_3,X_4\}^{co}.
$$  
Furthermore if in addition $\det(X_2-X_4)=0$, then $K^{lc}=K^{co}$.
\item If $\det(X_1-X_3)$ and $\det(X_2-X_4)$ have the same sign, then 
$$
K^{rc}=K^{lc,1}=[X_1,X_2]\cup[X_2,X_3]\cup[X_3,X_4]\cup[X_4,X_1].
$$
\item If $\det(X_1-X_3)$ and $\det(X_2-X_4)$ have opposite sign, then
$$
K^{pc}=K^{lc,2},
$$
and there exists a continuous increasing function 
$f:[0,1]\to[0,1]$ with $f(0)=0$ and $f(1)=1$ such that 
for any $t\in[0,1]$ we have
$$
\det\left((tX_1+(1-t)X_2)-(sX_4+(1-s)X_3)\right)=0,
$$
where $s=f(t)$.
\end{enumerate}
\end{lemma}

\begin{remark}
In fact, it is not difficult to check that if $X_i\in\R^{2\times 2}$, $i=1\dots 4$, are not coplanar, then there 
exists $R\in \textrm{span}\{X_1,\dots,X_4\}$ in the affine span and $\alpha\in\R$ such that for all $i$
\begin{equation}\label{e:smallsystem}
\det(X_i-R) = \alpha.
\end{equation}
To see this, subtract equation $i$ from equation $1$ in \eqref{e:smallsystem} and use \eqref{e:det} to obtain
\begin{equation*}
\langle \cof X_i-\cof X_1, R\rangle = \det X_i-\det X_1,\quad i=2,3,4.
\end{equation*}
Since we assumed that the matrices $X_1,\dots,X_4$ are not coplanar, the above linear system 
uniquely determines $R\in \textrm{span}\{X_1,\dots,X_4\}$. The scalar $\alpha$ is then obtained by back-substitution. 

Applying this observation to Case 3.~in the Lemma above shows that $X_1,\dots,X_4$ in this case lie on a one-sheeted hyperboloid (a doubly ruled surface) given by the equation $\det(X-R)=\alpha$.
\end{remark}

\begin{proof}[Proof of Lemma \ref{L:quadraticsurface}]
Let $P=\sum_{i=1}^4\lambda_iX_i$, where $\sum_{i=1}^4\lambda_i=1$.
By direct calculation 
\begin{equation*}
\begin{split}
\det P&=\sum_{i=1}^4\lambda_i\det X_i-\frac{1}{2}\sum_{i,j=1}^4
\lambda_i\lambda_j\det(X_i-X_j)\\
&=\sum_i\lambda_i\det X_i-\frac{1}{2}
(\lambda_1\lambda_3 d_{13}+\lambda_2\lambda_4d_{24}),
\end{split}
\end{equation*}
where $d_{ij}=\det(X_i-X_j)$.
In particular $P\in K^{pc}$ if and only if 
\begin{equation}\label{E:pc}
\lambda_1\lambda_3d_{13}+\lambda_2\lambda_4d_{24}=0.
\end{equation}
So if 
$d_{13}$ and $d_{24}$ have the same sign, then 
$$
K^{pc}=K^{lc,1}=[X_1,X_2]\cup[X_2,X_3]\cup[X_3,X_4]\cup[X_4,X_1].
$$
Suppose that $d_{13}>0$ and $d_{24}<0$, and let $P\in K^{pc}$. 
Consider the points 
$$
P_1=\frac{\lambda_1}{\lambda_1+\lambda_2}X_1+
\frac{\lambda_2}{\lambda_1+\lambda_2}X_2,\,
P_2=\frac{\lambda_3}{\lambda_3+\lambda_4}X_3+
\frac{\lambda_4}{\lambda_3+\lambda_4}X_4.
$$
Clearly $P\in[P_1,P_2]$. Furthermore
\begin{equation*}
\begin{split}
\det&(P_1-P_2)=\det(P_1)+\det(P_2)-\langle\cof P_1,P_2\rangle\\
=&\frac{\lambda_1}{\lambda_1+\lambda_2}\det X_1+
\frac{\lambda_2}{\lambda_1+\lambda_2}\det X_2+
\frac{\lambda_3}{\lambda_3+\lambda_4}\det X_3+
\frac{\lambda_4}{\lambda_3+\lambda_4}\det X_4-\\
&-\frac{1}{(\lambda_1+\lambda_2)(\lambda_3+\lambda_4)}
\biggl(\lambda_1\lambda_3(\det X_1+\det X_3-\det(X_1-X_3))+\\
&\quad\lambda_2\lambda_4(\det X_2+\det X_4-\det(X_2-X_4))+\\
&\quad\lambda_1\lambda_4(\det X_1+\det X_4)+
\lambda_2\lambda_3(\det X_2+\det X_3)\biggr)\\
=&\frac{\lambda_1}{\lambda_1+\lambda_2}\det X_1+
\frac{\lambda_2}{\lambda_1+\lambda_2}\det X_2+
\frac{\lambda_3}{\lambda_3+\lambda_4}\det X_3+
\frac{\lambda_4}{\lambda_3+\lambda_4}\det X_4-\\
&-\frac{1}{(\lambda_1+\lambda_2)(\lambda_3+\lambda_4)}
\biggl((\lambda_1+\lambda_2)(\lambda_3\det X_3+\lambda_4\det X_4)+\\
&\quad(\lambda_3+\lambda_4)(\lambda_1\det X_1+\lambda_2\det X_2)
\biggr)\\
=&0
\end{split}
\end{equation*}
using (\ref{E:pc}). Thus if $P\in K^{pc}$, then $P$ lies on the rank-one segment $[P_1,P_2]$, hence in $K^{lc,2}$. For completeness we find now $f$. Let $t\in (0,1)$ and $s=f(t)$. By above we have
$$
t=\frac{\lambda_1}{\lambda_1+\lambda_2}\,\text{ and }
s=\frac{\lambda_4}{\lambda_3+\lambda_4}.
$$
Thus $\lambda_2=\frac{1-t}{t}\lambda_1$ and 
$\lambda_3=\frac{1-s}{s}\lambda_4$. Substituting this into (\ref{E:pc}) gives
$$
(1-s)td_{13}+(1-t)sd_{24}=0
$$
which after rearranging gives
$$
f(t)=\frac{td_{13}}{td_{13}-(1-t)d_{24}}.
$$

\smallskip

If $\det(X_1-X_3)=0$, then we have the following degenerate cases. Either $K$ lies in a plane, which then is a rank-one plane (hence $K^{pc}=K^{lc,1}=K^c$),
or we have $\det(X_2-X_4)\neq 0$ and then condition (\ref{E:pc}) reduces to
$\lambda_2\lambda_4=0$, so that $K^{pc}=K^{lc,1}$ consists of the convex hulls of the two triangles $X_1,X_2,X_3$ and $X_1,X_3,X_4$.
\end{proof}


\section{Symmetric laminates on cubes}

In this section we prove our main theorem, Theorem \ref{t:main}. 
Let us recall the setting.
Let $C_1,C_2,C_3\in \R^{2\times 2}$ be rank-one matrices, and for $\epsilon\in\{-1,+1\}^3$ let 
$$
X_{\epsilon}=\sum_{i=1}^3\epsilon_iC_i
$$ 
and $K=\{X_{\epsilon}:\,\epsilon\in\{-1,+1\}^3\}$, see Figure \ref{f:rank1cube}.

\begin{figure}[H]
\begin{center}
\input{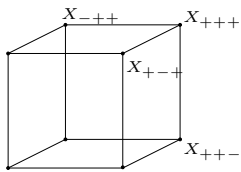}
\end{center}
\caption{The rank-one cube.}
\label{f:rank1cube}
\end{figure}

From now on we work in the coordinates given by $\{C_1,C_2,C_3\}$, i.e. $(x,y,z)$ corresponds to $xC_1+yC_2+zC_3$. The determinant in these coordinates is
$$
\det(x,y,z)=axy+bxz+cyz,
$$
being a quadratic form vanishing on the axes, and we have 
\begin{equation*}
\begin{split}
\det(1,1,1)&=a+b+c\\
\det(-1,1,1)&=-a-b+c\\
\det(1,-1,1)&=-a+b-c\\
\det(1,1,-1)&=a-b-c
\end{split}
\end{equation*}

\begin{lemma}
If $abc=0$, then any symmetric measure  in the sense of Definition \ref{d:symmetric} is a laminate.
\end{lemma}

\begin{proof}
We may assume without loss of generality that $a=0$. Then the $\{z=0\}$ plane is a rank-one plane, and symmetric measures on $K$ may be obtained by splitting (c.f. Definition \ref{d:laminate}) as follows:
\begin{align*}
\delta_{(0,0,0)}&\mapsto \frac{1}{2}\delta_{(0,0,1)}+\frac{1}{2}\delta_{(0,0,-1)}\\
&\mapsto \frac{1}{2}\left(\frac{1}{2}\delta_{(1,1,1)}+\frac{1}{2}\delta_{(-1,-1,1)}\right)+\frac{1}{2}\left(\frac{1}{2}\delta_{(1,1,-1)}+\frac{1}{2}\delta_{(-1,-1,-1)}\right).
\end{align*}
This is a symmetric laminate in the sense of Definition \ref{d:symmetric} with $\alpha=1/4$ and $\beta=0$. In a similar way we can obtain a symmetric laminate with $\alpha=0$ and $\beta=1/4$. Since symmetric laminates form a convex set and every symmetric measure can be written as a convex combination of these two laminates, we are done.
\end{proof}

In light of the preceding lemma, in the following we will assume $a,b,c\neq 0$.
Furthermore, by swapping the signs $\pm x,y,z$ and multiplying $X_i$ by $J$ if necessary, we can assume without loss of generality that
\begin{equation}\label{abc}
a,b,c>0.
\end{equation}

\begin{remark}
We note in passing that the assumption \eqref{abc} in particular implies that $\{C_1,C_2,C_3\}$ is linearly independent. Indeed, if 
$xC_1+yC_2=C_3$, then $0=\det C_3=axy$ and consequently $x=0$ or $y=0$. If (without loss of generality) $y=0$, then $0=\det(xC_1-C_3)=-bx$, hence $x=0$. 
\end{remark}

Our aim in this section is to construct symmetric laminates supported on the vertices of the cube $K$. The following lemma, which gives a simpler criterion for laminates to be symmetric, will be useful. 
\begin{lemma}\label{l:symmetric}
If $\nu$ is a laminate supported on $K$ with barycenter $\bar{\nu}=0$, and if 
$$
\nu_{+--}=\nu_{-+-}=\nu_{--+},
$$
then $\nu$ is symmetric.
\end{lemma}

\begin{proof}
Let us set $\nu_{---}=\beta$ and $\nu_{+--}=\nu_{-+-}=\nu_{--+}=\alpha$ (cf. Figure \ref{f:symmetric}).
Observe that, since $X\mapsto \pm\det X$ is rank-one convex, any laminate $\nu$ supported on $K$ and with barycenter $\bar{\nu}=0$ satisfies the equations \eqref{e:pc-constraints} with $N=3$. The first equation in  \eqref{e:pc-constraints}  amounts to
\begin{align*}
\nu_{+++}+\nu_{++-}+\nu_{+-+}+\beta&=2\beta+\alpha+\nu_{-++}\\
\nu_{+++}+\nu_{-++}+\nu_{++-}+\beta&=2\beta+\alpha+\nu_{+-+}\\
\nu_{+++}+\nu_{-++}+\nu_{+-+}+\beta&=2\beta+\alpha+\nu_{++-}.
\end{align*}
These three equations quickly lead to
$$
\nu_{++-}=\nu_{+-+}=\nu_{-++}=:\gamma.
$$
Then, the second equation in \eqref{e:pc-constraints} gives
$$
(a+b+c)(\alpha+2\beta-\gamma)+(\alpha+\gamma)(-a-b-c)=0,
$$
hence $(a+b+c)(\beta-\gamma)=0$. Since $a,b,c>0$, this implies $\gamma=\beta$. 
Back-substitution then yields $\nu_{+++}=\alpha$. 
\end{proof}

Next, we look for special laminates supported on minimal subsets of $K$. It will be convenient
from now on to switch the notation and  
write $X_0=(1,1,1)$, $X_1=(-1,1,1)$, $X_2=(1,-1,1)$ and $X_3=(1,1,-1)$. 
For later use we record
\begin{equation*}
\begin{split}
\det (X_1-X_2)&=-4a,\,\det (X_1-X_3)=-4b,\,\det (X_2-X_3)=-4c\\
\det(X_1+X_0)&=4c,\,\det(X_2+X_0)=4b,\,\det(X_3+X_0)=4a.
\end{split}
\end{equation*}

\medskip

We start with the following observation.
\begin{lemma}\label{lem:P}
Suppose $a,b,c>0$ and in addition $c<a+b$, i.e.~$\det X_1<0$. 
Then there exists $P\in [-X_0,-X_1]$ such that $\det(P-X_1)=0$, and 
furthermore
$$
0\in\{X_0,X_1,X_2,X_3,P\}^{lc}.
$$
\end{lemma}
\begin{figure}[H]
\begin{center}
\input{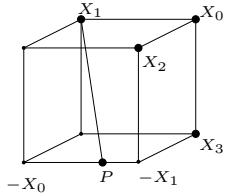}
\end{center}
\caption{The point $P\in [-X_0,-X_1]$.}
\label{f:pointP}
\end{figure}
\begin{proof}
Observe that the statement is symmetric with respect to swapping $X_2$ and $X_3$, i.e. with respect to swapping $a,b$. Therefore we may assume without loss of generality that 
$$
a\leq b.
$$ 
Let $P=\lambda(-X_0)+(1-\lambda)(-X_1)$. Then
$$
0=\lambda\det(-X_0-X_1)+(1-\lambda)\det(-X_1-X_1)=
4c\lambda +4(c-a-b)(1-\lambda)
$$
hence
\begin{equation}\label{lambda}
\lambda=\frac{a+b-c}{a+b}.
\end{equation}
See Figure \ref{f:pointP}.
Observe that $\det(P-X_0)=4(a+b)$. Now because $\det(X_1-X_2)=-4a<0$ and 
$$
\det(P-X_2)=\lambda\det(-X_0-X_2)+(1-\lambda)\det(-X_1-X_2)=4\lambda b>0,
$$
the point $P_1=\lambda_1X_1+(1-\lambda_1)P$ defned by
$$
\lambda_1=\frac{\lambda b}{a+\lambda b}
$$
is a point on the segment $[X_1,P]$ with $\det (P_1-X_2)=0$. 
Similarly $P_2=\lambda_2X_1+(1-\lambda_2)P$, where
$$
\lambda_2=\frac{\lambda a}{b+\lambda a}.
$$
Since we assumed that $a\leq b$, we have $\lambda_1\geq \lambda_2$. Then
$$
\det(X_2-P_2)=\frac{4\lambda(b^2-a^2)}{b+\lambda a}>0
\text{ and }
\det(X_2-X_3)=-4c
$$
so that there exists $P_3\in [P_2,X_3]$ with $\det(P_3-X_2)=0$.

In fact a simple calculation shows that
$P_3=\lambda_3X_3+(1-\lambda_3)P_2$, where
$$
\lambda_3=\frac{(a+b-c)(b-a)}{b^2-a^2+(1+\lambda)ac}.
$$
In particular if $a=b$, then $P_1=P_2=P_3$. Let us summarize so far. By construction we have
$P_1,P_2\in K^{lc,1}$, $P_3\in K^{lc,2}$ and furthermore
$$
\det(X_0-P_k)>0\text{ and }\det(X_i-X_j)<0,
$$
and
$$
\det(X_2-P_2)>0\text{ and }\det(P_1-P_3)<0
$$
(we get these by using repeatedly that the determinant is linear when restricted to rank-one lines). 

Now we can find the third lamination hull using Lemma \ref{L:quadraticsurface}. 
Indeed, for any of the following 4-tuples
\begin{equation*}
\{X_0,X_2,P_3,X_3\},\,\{X_0,X_1,P_1,X_2\},\,\{X_0,X_1,P_2,X_3\},\,\{X_2,P_1,P_2,P_3\}
\end{equation*}
case 3.~of Lemma \ref{L:quadraticsurface} applies and yields a ``filling'' of the corresponding rank-one square with
doubly ruled surfaces $\mathcal{S}_1,\dots,\mathcal{S}_4$, which are contained in the lamination-convex hull. See Figure \ref{f:Si}. Observe that any two such surfaces may intersect only along the common rank-one edge. To see this, consider for definiteness 
$X\in\mathcal{S}_1\cap\mathcal{S}_2$ and assume that $X\notin [X_0,X_2]$. In Lemma \ref{L:quadraticsurface} we showed that there exists a unique $R\in [X_0,X_2]$ with $\det(X-R)=0$, and there exist $Q_1\in [P_3,X_3]$, $Q_2\in [P_1,X_1]$ such that $X\in [Q_1,R]\cap [Q_2,R]$. But 
then either $Q_1\in [Q_2,R]\subset \mathcal{S}_2$ or $Q_2\in [Q_1,R]\subset \mathcal{S}_1$. Either case leads to a contradiction, as 
$[P_3,X_3]\cap [P_1,X_1]=\emptyset$. The intersection of any other pair $\mathcal{S}_i\cap\mathcal{S}_j$ can be handled in a similar fashion.

\begin{figure}[H]
\begin{center}
\input{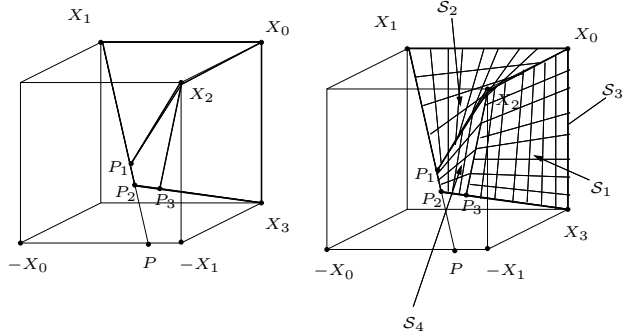}
\end{center}
\caption{The surface $\mathcal{S}=\bigcup_i\mathcal{S}_i$ and its 2-cell embedding.}
\label{f:Si}
\end{figure}

Consequently $\mathcal{S}=\bigcup_{i=1}^4\mathcal{S}_i$ is a regular (i.e. embedded) compact, piecewise smooth surface without boundary, with the rank-one edges of the surfaces $\mathcal{S}_i$ forming a canonical 2-cell embedding. It follows that therefore $\mathcal{S}$ is a topological sphere. By the Jordan-Brouwer theorem $\R^3\setminus\mathcal{S}$ consists of precisely two connected components, an ``inside'' $\mathcal{U}$ and an ``outside'' . Obviously any point $Q\in \mathcal{U}$ is contained on a rank-one segment connecting points on $\mathcal{S}=\partial\mathcal{U}$, therefore $\overline{\mathcal{U}}\subset K^{lc}$. Hence, 
in order to complete the proof, we need
to show that $0\in \mathcal{U}$.

Note that, by construction, the surface $\mathcal{S}$ depends continuously on the parameters $a,b,c$ in the set 
\begin{equation}\label{e:assumptions}
a,b,c>0,\quad c< a+b,\quad a\leq b.
\end{equation}
Furthermore, it is easy to check that in the case $a=b=c$, we have 
$P_1=P_2=P_3=\frac{-1}{3}(1,1,1)$. Hence in this case the line $(t,t,t)$ $(t\in\R)$ intersects $\mathcal{S}$ in precisely two points: $X_0=(1,1,1)$ and $P_1$, and consequently $0\in \mathcal{U}$. 
In fact it is not difficult to check that in this case $\overline{\mathcal{U}}=\left\{X_0, X_1, X_2, X_3, P_2\right\}^{pc}$, and we have
\begin{align*}
\lambda_0X_0+\lambda_1X_1+\lambda_2X_2+\lambda_3X_3+\lambda_4P_2&=0\\
\lambda_0\det X_0+\lambda_1\det X_1+\lambda_2\det X_2+\lambda_3\det X_3+\lambda_4\det P_2&=0\\
\end{align*}
with $\lambda_0 = \frac{9}{16}$, $\lambda_1 = \frac{1}{8}$, $\lambda_2 = \frac{1}{16}$, $\lambda_3 = \frac{1}{8}$, $\lambda_4 = \frac{1}{8}$ , showing that $0\in \mathcal{U}=\textrm{int}\left\{X_0, X_1, X_2, X_3, P_2\right\}^{pc}$ (this should be compared with calculations made in \cite{Pedregal:1998vt}).

Now let us find the instances when $0$ is on the boundary (in other words $0\in\mathcal{S}_k$ for some $k=1,2,3,4$). We may assume that $\lambda>0$, otherwise $0\in[X_1,P]$. But then, as $\mathcal{S}_2$ lies above the plane $z=y$ (more precisely, $z-y\geq 0$ for all $(x,y,z)\in \mathcal{S}_2$ with equality $z=y$ only if $(x,y,z)\in [X_0,X_1]\cup [X_1,P_1]$), it cannot contain $0$. Similarly, $0\notin\mathcal{S}_3$.

Suppose $0\in\mathcal{S}_4$. Then there exists, by definition of $\mathcal{S}_4$ (c.f. Lemma \ref{L:quadraticsurface}) $Q\in [P_1,P_2]$ 
with $\det(Q)=0$, and moreover, $Q$ is the unique point on the line segment $[X_1,P]$ with this property (since $\det(X_1-P)=0$). 
On the other hand we can calculate that $\frac{X_1+P}{2}=\frac{\lambda}{2}(X_1-X_0)$, so that $\det(\tfrac{X_1+P}{2})=0$. Therefore $Q=\tfrac{X_1+P}{2}$. From Lemma \ref{L:quadraticsurface} we also know that the rank-one line containing the segment $[\frac{X_1+P}{2},0]$ also needs to intersect $[X_2,P_3]$. In particular the orthogonal projection of $[X_2,P_3]$ onto the $\{x=0\}$ plane contains the origin. It is a simple matter to check that this is only possible if either $P_3=X_3$ (in which case $\det(X_2-X_3)=0$, i.e. $c=0$), or $P_2=\frac{X_1+P}{2}$. In the latter case $\lambda_2=1/2$, hence -- since we assumed $a\leq b$ --, $\lambda=1$ and $a+b=c$, contradicting our assumptions above. 

Finally, let us look at what happens if $0\in\mathcal{S}_1$. Then $0$ is on a rank-one segment $[Q_1,Q_2]$ connecting $[X_0,X_2]$ to $[X_3,P_2]$, see Figure \ref{f:finalcase}.

\begin{figure}[H]
\begin{center}
\input{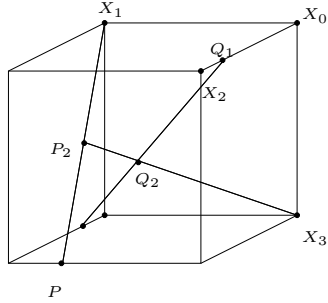}
\end{center}
\caption{The case $0\in \mathcal{S}_1$.}
\label{f:finalcase}
\end{figure}
Here $Q_1=\tilde\lambda X_0+(1-\tilde\lambda)X_2$ with 
$\tilde\lambda=\frac{a+c-b}{2(a+c)}$, and 
$$
Q_2=tX_3+(1-t)P_2=-sQ_1.
$$
Solving first for the first and third coordinate of $Q_2$ we obtain
$$
s=\frac{b}{2a+b}\text{ and }t=\frac{(a+b)(a-c)}{(a+b-c)(2a+b)}
$$
and substituting into the second coordinate we obtain 
$$
\frac{a^2-b^2-c^2}{(a+c)(a+b-c)}=0.
$$
But we assumed that $a\leq b$, hence this forces $c=0$ and $a=b$ again. 

We can conclude that under the assumptions \eqref{e:assumptions} the origin $0$ cannot lie on any one of the quadratic surfaces $\mathcal{S}_1,\dots,\mathcal{S}_4$. Consequently, for any $a,b,c$ satisfying \eqref{e:assumptions}, $0$ is contained in the ``inside'' $\mathcal{U}$, hence in the interior of the lamination-convex hull of $\{X_0,X_1,X_2,X_3,P\}$. This concludes the proof. 
\end{proof}

We are now ready to prove our main theorem, which we restate for the convenience of the reader. In the following it will be convenient to use the notation, that, given a probability measure $\nu$ supported on $$
K=\{X_{\epsilon}:\,\epsilon\in\{-1,+1\}^3\},
$$ 
we denote the mass of each point $X_{\epsilon}$ by $\nu(X_{\epsilon})$.

\begin{theorem}
The set $K$ supports symmetric laminates with barycenter $0$ and ratio
$$
\frac{\nu(X_0)}{\nu(X_1)}\leq \frac{1}{3}.
$$
In particular, given any $C_1,C_2,C_3$ rank-1 matrices in $\R^{2\times 2}$, for any rank-one convex function $f:\R^{2\times 2}\to\R$ we have
\begin{equation*}
\begin{split}
f(0)\leq&\frac{1}{16}\bigl(f(X_{+++})+f(X_{+--})+f(X_{-+-})+f(X_{--+})\bigr)+\\
&\frac{3}{16}\bigl(f(X_{---})+f(X_{++-})+f(X_{+-+})+f(X_{-++})\bigr),
\end{split}
\end{equation*}
where $X_{\pm\pm\pm}$ denotes the matrix $\pm C_1\pm C_2\pm C_3$.
\end{theorem}

\begin{proof}
As in the beginning of this section, we may assume without loss of generality that $a,b,c>0$. In particular this means that $\det X_0>0$, and then there are two cases depending on the signs of $\det X_i$ for $i=1,2,3$.
\begin{enumerate}
\item $\det X_i<0$ for all $i$,
\item $\det X_i>0$ for some $i$.
\end{enumerate}

\noindent{\bf Case 1.}
If $\det X_1,\det X_2,\det X_3<0$, corresponding to 
$$
a+b>c,\,\, a+c>b\,\text{ and } b+c>a,
$$
then Lemma \ref{lem:P} implies the existence of $P_1,P_2,P_3$ on the segments
$[-X_0,-X_1]$, $[-X_0,-X_2]$, $[-X_0,-X_3]$ respectively, with the property that
$$
0\in\{X_0,X_1,X_2,X_3,P_k\}^{lc}.
$$ 
Let $\tilde\nu_1$ be a laminate supported on $\{X_0,X_1,X_2,X_3,P_1\}$ with barycenter $0$. As $P_1$ lies on the rank-one segment $[-X_0,-X_1]$, more precisely (c.f.~(\ref{lambda}))
$$
P_1=\lambda(-X_0)+(1-\lambda)(-X_1)
$$
with $\lambda=\frac{a+b-c}{a+b}$, we can split along this segment 
to obtain the laminate $\nu_1$ supported on the set
$$
K_1=\{X_0,X_1,X_2,X_3,-X_0,-X_1\}
$$
with barycenter $0$. Furthermore,
$$
\frac{\nu_1(-X_0)}{\nu_1(-X_1)}=\frac{\lambda}{1-\lambda}=\frac{a+b-c}{c}.
$$

In a similar manner we obtain the laminates 
$\nu_2$ and $\nu_3$ with barycenter $0$, supported on $K_2$ and $K_3$ respectively, where
$$
K_i=\{X_0,X_1,X_2,X_3,-X_0,-X_i\},
$$
and such that 
$$
\frac{\nu_2(-X_0)}{\nu_2(-X_2)}=\frac{a+c-b}{b}\,\text{ and }
\frac{\nu_3(-X_0)}{\nu_3(-X_3)}=\frac{b+c-a}{a}.
$$
To obtain a symmetric laminate, we form the convex combination:
$$
\nu=C\bigl(\frac{1}{\nu_1(-X_1)}\nu_1+\frac{1}{\nu_2(-X_2)}\nu_2
+\frac{1}{\nu_3(-X_3)}\nu_3\bigr)
$$
where $C>0$ is the normalizing factor
$$
C=\frac{1}{\frac{1}{\nu_1(-X_1)}+\frac{1}{\nu_2(-X_2)}+\frac{1}{\nu_3(-X_3)}}.
$$
Note that $\nu(-X_i)=C$ for all $i=1,2,3$, hence $\nu$ is symmetric in light of Lemma \ref{l:symmetric}. Furthermore
\begin{equation*}
\begin{split}
\frac{\nu(-X_0)}{\nu(-X_k)}&=\sum_{i=1}^3\frac{\nu_i(-X_0)}{\nu_i(-X_i)}\\
&=\frac{a+b}{c}+\frac{a+c}{b}+\frac{b+c}{a}-3\\
&=(a+b+c)\left(\frac{1}{a}+\frac{1}{b}+\frac{1}{c}\right)-6\\
&\geq 9-6=3,
\end{split}
\end{equation*}
where in the last line we used the harmonic-arithmetic mean inequality. In particular we see that 
$$
\min_{a,b,c>0}\frac{\nu(-X_0)}{\nu(-X_k)}=
\frac{\nu(-X_0)}{\nu(-X_k)}\biggl|_{a=b=c}=3,
$$
so that 
$$
\nu(-X_0)\geq 3\nu(-X_k).
$$

\smallskip

\noindent{\bf Case 2.}
Suppose $\det X_2>0$, i.e.~$b>a+c$. Then in particular, as $a,b,c>0$, necessarily $a+b>c$ and $b+c>a$, so that $\det X_1<0$ and $\det X_3<0$. Because $\det X_2>0$, there is no rank-one line from $X_2$
hitting the segment $[-X_0,-X_2]$, but instead there is one hitting the segment
$[X_1,-X_2]$. Indeed, setting $P=\lambda X_1+(1-\lambda)(-X_2)$ and requiring $\det(P-X_2)=0$, we obtain
$$
0=4(b-a-c)-4\lambda (b-c),
$$
hence $\lambda=\frac{b-a-c}{b-c}$. Note that $b-c>a$ by assumption, so that $\lambda\in (0,1)$.

\begin{figure}[H]
\begin{center}
\input{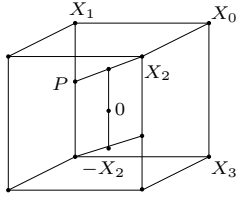}
\end{center}
\caption{Obtaining $\nu_2'$ in Case 2.}
\label{f:case2}
\end{figure}

Then splitting first along the line $(0,0,1)$, then along the parallel 
segments $[P,X_2]$ and $[-X_2,-P]$ and finally on the segments
$[-X_2,X_1]$ and $[-X_1,X_2]$ (see Figure \ref{f:case2}) we obtain the laminate $\nu_2'$ with barycenter $0$ and supported on
$$
\{X_1,X_2,-X_1,-X_2\}.
$$
Furthermore, a quick calculation shows that 
\begin{align*}
\nu_2'(X_1)=\nu_2'(-X_1)&=\frac{1}{4}\lambda,\\ 
\nu_2'(X_2)=\nu_2'(-X_2)&=\frac{1}{4}(2-\lambda),
\end{align*} 
with $\lambda=\frac{b-a-c}{b-c}$ from above.
In addition we have the laminates $\nu_1$ and $\nu_3$ as before. Our symmetric
laminate this time will be
$$
\nu'=C\left(\frac{1}{\nu_1(-X_1)}\left(1-\frac{\nu_2'(-X_1)}{\nu_2'(-X_2)}\right)\nu_1+
\frac{1}{\nu_2'(-X_2)}\nu_2'+\frac{1}{\nu_3(-X_3)}\nu_3\right),
$$
where again $C>0$ is a normalizing factor (so that $\nu'$ is a probability measure). Since 
$\nu_2'(-X_1)/\nu_2'(-X_2)=\lambda/(2-\lambda)\leq 1$, $\nu'$ is a probability measure. 
Note that 
$$
\nu'(-X_1)=\nu'(-X_2)=\nu'(-X_3)=C,
$$
hence $\nu'$ is symmetric by Lemma \ref{l:symmetric}.
Moreover, for $k=1,2,3$ 
$$
\frac{\nu'(-X_0)}{\nu'(-X_k)}=\left(1-\frac{\nu_2'(-X_1)}{\nu_2'(-X_2)}\right)
\frac{\nu_1(-X_0)}{\nu_1(-X_1)}
+\frac{\nu_3(-X_0)}{\nu_3(-X_3)}
$$
since $\nu_2'(-X_0)=0$. Substituting and using that $b>a+c$ we get
$$
\frac{\nu'(-X_0)}{\nu'(-X_k)}=\frac{2a}{c}+\frac{b+c-a}{a}>
\frac{2a}{c}+\frac{2c}{a}\geq 4.
$$
Therefore the symmetric laminate $\nu'$ we obtain in this case also 
satisfies
$\nu'(-X_0)\geq 3\nu'(-X_k)$.
\end{proof}

\section{Acknowledgement}
This work was done during the visit of the second author at the NumNet Research Group at the E\"otv\"os L\'or\'and University in Budapest as part of the Guest Scientist Visiting Programme of the Hungarian Academy of Sciences. He acknowledges the support of the Hungarian Academy of Sciences and the hospitality of the Institute of Mathematics at the E\"otv\"os L\'or\'and University.


\bibliographystyle{acm}

\end{document}